\documentclass[a4paper,10pt]{amsart}
\usepackage[utf8]{inputenc}
\usepackage[T2A]{fontenc}
\usepackage[english]{babel}

\usepackage{amssymb}
\usepackage{amsthm}

\newtheorem{lemma}{Lemma}
\newtheorem{theorem}{Theorem}

\def\hm#1{#1\nobreak\discretionary{}{\hbox{\ensuremath{#1}}}{}}

\def\id{\mathop{\mathrm{id}}}

\def\mue{\mu^{(e)}}
\def\phie{{\mathfrak f}^{(e)}}

\def\taue{\tau^{(e)}}
\def\taui{\tau^{\infty}}
\def\tauei{\tau^{(e)\infty}}

\def\sigmae{\sigma^{(e)}}

\def\sigmai{\sigma^{\infty}}
\def\sigmaei{\sigma^{(e)\infty}}

\def\eps{\varepsilon}

\def\res{\mathop{\mathrm{res}}}
\def\llog{\mathop{\mathrm{llog}}\nolimits}
\def\lllog{\mathop{\mathrm{lllog}}\nolimits}

\def\le{\leqslant}
\def\ge{\geqslant}

\def\oeis#1{\href{http://oeis.org/#1}{#1}}

\usepackage[unicode]{hyperref}
\usepackage{breakurl}

\begin{document}

\title{Exponential and infinitary divisors}
\author{Andrew V. Lelechenko}
\address{I.~I.~Mechnikov Odessa National University}
\email{1@dxdy.ru}

\keywords{Exponential divisors, infinitary divisors, modified sum-of-divisors function, modified totient function}
\subjclass[2010]{
11A25,  
11N37, 
11N56
}

\begin{abstract}
Our paper is devoted to several problems from the field of modified divisors: namely exponential and infinitary divisors. We study the behaviour of modified divisors, sum-of-divisors and totient functions. Main results concern with the asymptotic behaviour of mean values and explicit estimates of extremal orders.
\end{abstract}

\maketitle

\section{Introduction}\label{s:introduction}

Let $m$ be an {\em exponential divisor} (or {\em e-divisor}) of $n$ (denote $m \mid^{(e)} n$) if $m\mid n$ and for each prime $p\mid n$ we have $a\mid b$, where $p^a \,||\, m$, $p^b \,||\, n$. This concept, introduced by Subbarao \cite{subbarao1972}, leads us to the {\em e-divisor function} $\taue(n) = \sum_{m\mid^{(e)} n} 1$ (sequence \oeis{A049419} in OEIS~\cite{oeis2011}) and {\em sum-of-e-divisors function} $\sigmae(n) = \sum_{m\mid^{(e)} n} m$ (sequence \oeis{A051377}). These functions were studied by many authors, including among others Wu and Pétermann \cite{petermann1997, wu1995}.

Consider a set of arithmetic functions $\mathcal A$, a set of multiplicative prime-in\-de\-pen\-dent functions~$\mathcal M_{PI}$ and an operator $E\colon {\mathcal A} \to {\mathcal M}_{PI}$ such that
$$
(Ef)(p^a) = f(a).
$$
One can check that $\taue = E\tau$, but $\sigmae\ne E\sigma$. Section~\ref{s:E-sigma} is devoted to the latter new function~$E\sigma$.

On contrary several authors, including Tóth~\cite{toth2004, toth2007a} and Pé\-ter\-mann~\cite{petermann2010}, studied exponential analogue of the totient function, defining~$\phi^{(e)} = Ef$. However $\phi^{(e)}$ lacks many significant properties of $\phi$: it is pri\-me-in\-de\-pen\-dent and $\phi^{(e)} \ll n^\eps$. In Section~\ref{s:phie} we construct more natural modification of the totient function, which will be denoted by~$\phie$.

\medskip

One can define {\em unitary divisors} as follows: $m \mid^* n$ if $m\mid n$ and $\gcd(m,n/m)=1$. Further, define {\em bi-unitary divisors:} $m \mid^{**} n$ if $m \mid n$ and greatest common unitary divisor of $m$ and $n/m$ is 1; define {\em tri-unitary divisors:} $m \mid^{***} n$ if $m \mid n$ and greatest common bi-unitary divisor of $m$ and $n/m$ is 1; and so on. It appears that this process converges to the set of so-called {\em infinitary divisors} (or $\infty$-divisors): $m \mid^\infty n$ if $m\mid n$ and for each $p\mid n$, $p^a \,||\, m$, $p^b \,||\, n$, the binary digits of $a$ have zeros in all places, where $b$'s have. This notation immediately induces $\infty$-divisor function~$\taui$ (sequence \oeis{A037445}) and sum-of-$\infty$-divisors function $\sigmai$ (sequence \oeis{A049417}). See Cohen \cite{cohen1990}.

Recently Minculete and Tóth \cite{minculete2011d} defined and studied an exponential analogue of unitary divisors. We introduce e-$\infty$-divisors: $m \mid^{(e)\infty} n$ if $m\mid n$ and for each $p\mid n$, $p^a \,||\, m$, $p^b \,||\, n$, we have $a \mid^\infty b$. In Section~\ref{s:taui} we improve an estimate for $\sum_{n\le x} \taui(n)$ by Cohen and Hagis \cite{cohen1993} and briefly examine~$\tauei$. Section~\ref{s:perfect} is devoted to e-$\infty$-perfect numbers such that $\sigmaei(n) = 2n$.

\section{Notations}

Letter $p$ with or without indexes denotes a prime number. Notation~$p^a \,||\, n$ means that $p^a \mid n$, but $p^{a+1} \nmid n$.

We write $f\star g$ for Dirichlet convolution
$$ (f \star g)(n) = \sum_{d|n} f(d) g(n/d). $$

In asymptotic relations we use $\sim$, $\asymp$, Landau symbols $O$ and $o$, Vinogradov symbols $\ll$ and~$\gg$ in their usual meanings. All asymptotic relations are given as an argument (usually $x$) tends to the infinity.

Letter $\gamma$ denotes Euler---Mascheroni constant.  Everywhere $\eps>0$ is an arbitrarily small number (not always the same even in one equation).

As usual $\zeta(s)$ is the Riemann zeta-function.
Real and imaginary components of the complex~$s$ are denoted as $\sigma:=\Re s$ and~$t:=\Im s$, so~$s=\sigma+it$.

We abbreviate $\llog x := \log\log x$, $\lllog x := \log\log\log x$, where $\log x$ is a natural logarithm.

\medskip

Let $\tau$ be a divisor function, $\tau(n) = \sum_{d|n} 1$.
Denote
$$
\tau(a_1,\ldots,a_k; n) = \sum_{d_1^{a_1}\cdots d_k^{a_k} = n} 1
$$
and $\tau_k = \tau(\underbrace{1,\ldots,1}_{k\text{~times}}; \cdot)$.
Then $\tau \equiv \tau_2 \equiv \tau(1,1; \cdot)$.

Now let $\Delta(a_1,\ldots,a_k; x)$ be an error term in the asymptotic estimate of the sum~$\sum_{n\le x} \tau(a_1,\hm\ldots,a_k; n)$. (See \cite{kratzel1988} for the form of the main term.) For the sake of brevity denote $
\Delta_k(x) = \Delta(\underbrace{1,\ldots,1}_{k\text{~times}}; x)
$.

Finally, $\theta(a_1,\ldots,a_k)$ denotes throughout our paper a real value such that
$$ \Delta(a_1,\ldots,a_k; x) \ll x^{\theta(a_1,\ldots,a_k)+\eps} $$
and we write $\theta_k$ for the exponent of $x$ in $\Delta_k(x)$.


\section{Values of $E\sigma$}\label{s:E-sigma}

\begin{theorem}
\begin{equation}\label{eq:Esigma-limsup}
\limsup_{n\to\infty} {\log E\sigma (n) \llog n \over \log n} = {\log 3 \over 2}.
\end{equation}
\end{theorem}
\begin{proof}
Theorem of Suryanarayana and Sita Rama Chandra Rao \cite{suryanarayana1975} shows that
$$
\limsup_{n\to\infty} {\log E\sigma (n) \llog n \over \log n}
= \sup_{n\ge1} {\log \sigma(n) \over n}.
$$
The supremum can be split into two parts: we have
$$
\max_{n\le6} {\log\sigma(n) \over n}
= {\log\sigma(2) \over 2} = {\log3\over2}
$$
and for $n>6$ we apply estimate of Ivić~\cite{ivic1977}
\begin{equation}\label{eq:sigma-ivic}
\sigma(n) < 2.59 n \llog n
\end{equation}
to obtain
$$
{\log \sigma(n) \over n}
< {\log 2.59 + \log n + \lllog n \over n} := f(n),
$$
where $f$ is a decreasing function for $n>6$ and $f(7) < (\log 3) / 2$. Thus
$$
\sup_{n\ge 1} {\log\sigma(n) \over n} = {\log 3 \over  2}. $$
\end{proof}

Equation~\eqref{eq:Esigma-limsup} shows that $E\sigma(n) \ll n^\eps$.

\begin{theorem}
\begin{equation}\label{eq:E-sigma-sum}
\sum_{n\le x} E\sigma(n) = C_1 x + (C_2\log x + C_3) x^{1/2} + C_4 x^{1/3} + E(x),
\end{equation}
where $C_1$, $C_2$, $C_3$, $C_4$ are computable constants and
$$
x^{1/5} \ll E(x) \ll x^{1153/3613+\eps}.
$$
\end{theorem}
\begin{proof}
Let $F(s) = \sum_{n=1}^\infty E\sigma(n) n^{-s}$. We have utilizing~\eqref{eq:sigma-ivic}
\begin{multline*}
F(s)
= \prod_p \sum_{a=0}^\infty E\sigma(p^a) p^{-as}
= \prod_p \biggl( 1 + \sum_{a=1}^\infty \sigma(a) p^{-as} \biggr)
=\\=
 \prod_p \left( 1 + p^{-s} + 3p^{-2s} + 4p^{-3s} + 7p^{-4s} + O(p^{\eps-5s}) \right)
=\\=
 \prod_p {1+O(p^{\eps-5s}) \over (1-p^{-s})(1-p^{-2s})^2(1-p^{-3s})},
\end{multline*}
so
\begin{equation}\label{eq:E-sigma-series}
F(s) = \zeta(s) \zeta^2(2s) \zeta(3s) H(s),
\end{equation}
where series $H(s)$ converges absolutely for $\sigma>1/5$.

Equation~\eqref{eq:E-sigma-series} shows that
$$E\sigma = \tau(1,2,2,3;\cdot) \star h,$$ where $\sum_{n\le x} |h(n)| \ll x^{1/5+\eps}$. We apply the result of Krätzel~\cite[Th.~3]{kratzel1992} together with Huxley's~\cite{huxley2005} exponent pair $k=32/205+\eps$, $l=k+1/2$  to obtain
$$
\sum_{n\le x} \tau(1,2,2,3; n) = B_1 x + (B_2\log x + B_3) x^{1/2} + B_4 x^{1/3} + O(x^{1153/3613+\eps})
$$
for some computable constants $B_1$, $B_2$, $B_3$, $B_4$.
Now convolution argument certifies~\eqref{eq:E-sigma-sum} and the upper bound of~$E(x)$. The lower bound for~$E(x)$ follows from the theorem of Küh\-leit\-ner and Nowak \cite{kuhleitner1994}.
\end{proof}

\begin{theorem}
$$
\sum_{n\le x} (E\sigma(n))^2 = D x + P_7(\log x) x^{1/2} + E(x),
$$
where $D$ is a computable constant, $P_7$ is a polynomial with $\deg P = 7$ and
$$
x^{4/17} \ll E(x) \ll x^{8/19+\eps}.
$$
\end{theorem}
\begin{proof}
We have
\begin{multline*}
\sum_{n=1}^\infty (E\sigma(n))^2 n^{-s}
= \prod_p \biggl( 1 + \sum_{a=1}^\infty (\sigma(a))^2 p^{-as} \biggr)
=\\=
 \prod_p \left( 1 + p^{-s} + 9p^{-2s} + O(p^{\eps-3s}) \right)
=\\=
 \prod_p {1+O(p^{\eps-3s}) \over (1-p^{-s})(1-p^{-2s})^8}
= \zeta(s) \zeta^8(2s) G(s),
\end{multline*}
where series $G(s)$ converges absolutely for $\sigma>1/3$.

$\Omega$-estimate of the error term $E(x)$ follows again from \cite{kuhleitner1994}. To obtain~$E(x) \hm\ll x^{8/19}$ we use~\cite[Th.~6.8]{kratzel1988}, which implies
$$
\theta(1,2,2,2,2,2,2,2,2) \le
{1\over 1+2-\theta_8} \le {8\over 19}.
$$
Here we used the estimate of Heath-Brown $\theta_8 \le 5/8$ \cite[p. 325]{titchmarsh1986}.
\end{proof}

\section{Values of $\phie$}\label{s:phie}

For the usual Möbius function $\mu$, identity function $\id$ and unit function $\mathbf1$ we have
\begin{align*}
\tau &= \mathbf1 \star \mathbf1, \\
\id &= \mathbf1 \star \mu, \\
\sigma &= \mathbf1 \star \id.
\end{align*}

Subbarao introduced in \cite{subbarao1972} the exponential convolution $\odot$ such that for multiplicative $f$ and~$g$ their convolution $f\odot g$ is also multiplicative with
\begin{equation}\label{eq:exp-conv}
(f\odot g)(p^a) = \sum_{d|a} f(p^d) g(p^{a/d}).
\end{equation}
For function $\mue = E\mu$ and defined in Section~\ref{s:introduction} functions $\taue$ and $\sigmae$ we have
\begin{align*}
\taue &= \mathbf1 \odot \mathbf1, \\
\id &= \mathbf1 \odot \mue, \\
\sigmae &= \mathbf1 \odot \id.
\end{align*}

This leads us to the natural definition of $\phie=\mue\odot\id$ (similar to usual $\phi \hm= \mu\star\id$). Then by definition~\eqref{eq:exp-conv}
$$
\phie(p^a) = \sum_{d|a} \mu(a/d) p^d.
$$

Let us list a few first values of $\phie$ on prime powers:

\begin{table}[h]
\begin{tabular}{c|ccccc}
$a$ & 1 & 2 & 3 & 4 & 5 \\\hline
$\phie(p^a)$ & $p$ & $p^2-p$ & $p^3-p$ & $p^4-p^2$ & $p^5-p$
\end{tabular}
\end{table}

Note that $\phie(n)/n$ depends only on the square-full part of $n$.
Trivially
$$
\limsup_{n\to\infty} {\phie(n) \over n} = 1
$$
and one can show utilizing Mertens formula (cf.~\cite[Th.~328]{hardy2008}) that
$$
\liminf_{n\to\infty} {\phie(n) \llog n \over n} = e^{-\gamma}.
$$
Instead we prove an explicit result.

\begin{theorem}
For any $n>44100$
\begin{equation}\label{eq:phie-explicit}
{\phie(n) \llog n \over n} \ge C e^{-\gamma},
\qquad
C = 0.993957.
\end{equation}
\end{theorem}
\begin{proof}
Denote for brevity $ f(n) = \phie(n) \llog n / n $.

Let $s(n)$ count primes, squares of which divide $n$:
$ s(n) = \sum_{p^2 \mid n} 1 $.
Let $p_k$ denote the~$k$-th prime: $p_1=2$, $p_2=3$ and so on.
One can check that
$$
f(n) \le f\left( \prod_{p \le p_{s(n)}}  p^2 \right).
$$

\medskip

Now we are going to prove that for every $x\ge 11$ inequality~\eqref{eq:phie-explicit} holds for~$n\hm=\prod_{p\le x} p^2$. On such numbers we have
$$
f(n) = \prod_{p\le x} (1-p^{-1}) \cdot \log\left( 2 \sum_{p\le x} \log p \right)
$$
and our goal is to estimate the right hand side from the bottom. By Dusart~\cite{dusart1999} we know that for $ x \ge x_0 = 10\,544\,111 $
\begin{align}
\label{eq:explicit-p-1}
\sum_{p\le x} p^{-1} &\le \llog x + B + {1\over 10 \log^2 x} + {4\over 15\log^3 x}, \\
\label{eq:explicit-log-p}
\sum_{p\le x} \log p &\ge x\left( 1 - {0.006788 \over \log x} \right)
\end{align}

Now since $\exp(-y-y^2/2-cy^3) \le 1-y$ for $0\le y\le 1/2$ and~$c\hm=8\log 2-5 \hm= 0.545177$ we have
\begin{multline*}
\prod_{p\le x} (1-p^{-1})
\ge \prod_{p\le x} \exp(-p^{-1}-p^{-2}/2-c p^{-3})
\ge \\ \ge
\left( \prod_p \exp(-p^{-2}/2-c p^{-3}) \right) \prod_{p\le x} \exp(-p^{-1})
=: C_1 \prod_{p\le x}  \exp(-p^{-1}),
\end{multline*}
where $C_1=0.725132$. Further, by~\eqref{eq:explicit-p-1} for $x\ge x_0$
\begin{multline*}
\prod_{p\le x}  \exp(-p^{-1})
= \exp\left( - \sum_{p\le x} p^{-1} \right)
\ge \\ \ge \log^{-1} x \cdot \exp\left(-B - {1\over 10 \log^2 x} - {4\over 15\log^3 x}\right)
\ge \\ \ge \log^{-1} x \cdot \exp\left(-B - {1\over 10 \log^2 x_0} - {4\over 15\log^3 x_0}\right)
=:  C_2 \log^{-1} x,
\end{multline*}
where $C_2=0.769606$. And by~\eqref{eq:explicit-log-p} for $x\ge x_0$
$$
\log\left( 2 \sum_{p\le x} \log p \right)
\ge \log\left( 2 x \left( 1-{0.006788 \over \log x} \right) \right)
\ge \log x.
$$

Finally, we obtain that for $x\ge x_0$, $n=\prod_{p\le x} p^2$ we have
\begin{equation}\label{eq:c1-c2}
f(n) \ge C_1 C_2 = 0.993957 e^{-\gamma}.
\end{equation}
Numerical computations show that in fact \eqref{eq:c1-c2} holds for $p_5 \hm= 11 \hm\le x \hm< x_0$ and~$n \hm= (2\cdot3\cdot5\cdot7)^2 \hm= 44100$ is the largest exception of form~$\prod_{p\le x} p^2$.

\medskip

To complete the proof we should show that the theorem is valid for each $n \hm> 44100$ such that~$s(n)\le4$. Firstly, one can  validate that the only square-full numbers~$k$ for which $f(k) \ge Ce^{-\gamma}$ and $s(k)\le 4$ are 4, 8, 9, 36, 900, 44100. Secondly, let $n=kl$, where $k>1$ stands for square-full part and $l$ for square-free part, $\gcd(k,l)=1$. Then
$$
f(n) = {\phie(k) \llog k \over k} \cdot {\phie(l)\over l} \cdot {\llog kl \over \llog k}
\ge {2\llog 4 \over 4} \cdot {\llog 4l \over \llog 4} = {\llog 4l \over 2}.
$$
This inequality shows that if $f(n) \le C e^{-\gamma}$ then $\llog 4l \le 2 C e^{-\gamma}$ or equivalently~$l \hm\le 5$.

Thus the complete set of suspicious numbers is
$$
\bigl\{kl \mid k\in(4, 8, 9, 36, 900, 44100), l \in \{1,2,3,5\}, \gcd(k,l)=1 \bigr\}
$$
and fortunately all of them are less or equal to 44100.
\end{proof}

\begin{theorem}
$$
\sum_{n\le x} \phie(n) = C x^2 + O(x \log^{5/3} x),
$$
where $C$ is a computable constant.
\end{theorem}
\begin{proof}
Let $s$ be a complex number such that $\sigma > 4/5$. For $a\ge4$ one have $\phie(p^a) \hm= p^a+O(p^{a/2})$
and
$$
\sum_{a=4}^\infty p^{a/2-4a/5} = \sum_{a=4}^\infty p^{-3a/10} \ll p^{-12/10} \ll p^{-1}.
$$
We have
\begin{multline*}
\mathfrak{F}(p)
:= \sum_{a=0}^\infty \phie(p^a) p^{-as}
=\\= 1 + p^{1-s} + (p^{2-2s}-p^{1-2s}) + (p^{3-3s}-p^{1-3s}) + \sum_{a=4}^\infty p^{a-as} + O(p^{-1}).
\end{multline*}
Then
\begin{multline*}
(1-p^{1-s}) \mathfrak{F} (p) = 1 - p^{1-2s} + p^{2-3s} - p^{1-3s} + p^{2-4s} + O(p^{-1})
=\\= 1 - p^{1-2s} + p^{2-3s} + O(p^{-1})
\end{multline*}
and
$$
{ (1-p^{1-s}) (1-p^{2-3s}) \over 1-p^{1-2s} } \mathfrak{F} (p) = 1+O(p^{-1}).
$$
Taking product by $p$ we obtain
$$
\sum_{n=1}^\infty \phie(n) n^{-s}
= \prod_p \mathfrak{F}(p)
= {\zeta(s-1) \zeta(3s-2) \over \zeta(2s-1)} G(s),
$$
where $G(s)$ converges absolutely for $\sigma > 4/5$. This means that
$\phie \hm= z \star g$, where
$$
z(n) = \sum_{n_1 n_2^2 n_3^3 = n} n_1 \mu(n_2) n_2 n_3^2
$$
and $\sum_{n\le x} |g(n)| \ll x^{4/5+\eps}$.

By \cite[Th. 1]{petermann1997} we have $\sum_{n_1 n_3^3 \le y} n_1 n_3^2 = y^2  \zeta(4) / 2 + O(y \log^{2/3} y)$, so
\begin{multline*}
\sum_{n\le x} z(n)
 = \sum_{n_2 \le x^{1/2}} \mu(n_2) n_2 \left( {\zeta(4) \over 2} {x^2\over n_2^4} + O\left({x\over n_2^2} \log^{2/3} x\right) \right)
 =\\= {\zeta(4) \over 2 \zeta(3)} x^2 + O(x \log^{5/3} x).
\end{multline*}

Standard convolution argument completes the proof.
\end{proof}

\section{Values of $\taui$ and $\tauei$}\label{s:taui}

Note that $\taui(p)=\taui(p^2)=\taui(p^4)=2$, $\taui(p^3)=\taui(p^5)=4$ and more generally
\begin{equation}\label{eq:taui-primorial}
\taui(p^a) = 2^{u(a)},
\end{equation}
where $u(a)$ is equal to the number of units in binary representation of~$a$. Thus~$\taui(p^a) \hm\le a+1$ and $\taui(n) \ll n^\eps$.

\begin{theorem}
$$
\sum_{n\le x} \taui(n) = (D_1\log x + D_2) x + E(x),
$$
where $D_1$, $D_2$ are computable constants. In unconditional case
$$
E(x) \ll x^{1/2} \exp(-A \log^{3/5} x \llog^{-1/5} x), \qquad A>0,
$$
and under Riemann hypothesis
$
E(x) \ll x^{5/11+\eps}
$.
\end{theorem}
\begin{proof}
Let us transform Dirichlet series for $\taui$ into a product of zeta-functions:
\begin{multline*}
\sum_{n=1}^\infty \taui(n) n^{-s}
= \prod_p \sum_{a=0}^\infty \taui(p^a) p^{-as}
=\\= \prod_p \bigl( 1 + 2p^{-s} + 2p^{-2s} + 4p^{-3s} + O(p^{\eps-4s}) \bigr)
=\\= \prod_p {\bigl(1+O(p^{\eps-4s})\bigr) (1-p^{-2s}) \over (1-p^{-s})^2 (1-p^{-3s})^2}
= {\zeta^2(s) \zeta^2(3s) \over \zeta(2s)} G(s),
\end{multline*}
where series $G(s)$ converges absolutely for $\sigma > 1/4$.

By \cite[Th. 6.8]{kratzel1988} together with estimate $\theta_2 < 131/416+\eps$ from \cite{huxley2005} we get
$$
\sum_{n\le x} \tau(1,1,3,3;n) = (C_1 \log x + C_2) + (C_3 \log x + C_4) x^{1/3} + O(x^{547/1664+\eps}).
$$
Now the statement of the theorem can be achieved by application of  Ivić's \cite[Th.~2]{ivic1978}.
Alas, term $(C_3 \log x + C_4) x^{1/3}$ will be absorbed by error term.
\end{proof}

\begin{theorem}
\begin{equation}\label{eq:taui^2-sum}
\sum_{n\le x} \bigl( \taui(n) \bigr)^2
= P_3(\log x) x + O(x^{1/2} \log^9 x),
\end{equation}
where $P_3$ is a polynomial, $\deg P_3 = 3$.
\end{theorem}
\begin{proof}
We have
$$
\bigl( \taui(p) \bigr)^2 =
\bigl( \taui(p^2) \bigr)^2 = 4,
$$
so
$$
F(s) := \sum_{n=1}^\infty \bigl( \taui(n) \bigr)^2 n^{-s}
= {\zeta^4(s) \over \zeta^6(2s)} H(s),
$$
where series $H(s)$ converges absolutely for $\sigma > 1/3$.

By Perron formula for $c:=1+1/\log x$, $\log T \asymp \log x$ we have
$$
\sum_{n\le x} \bigl( \taui(n) \bigr)^2
= {1\over 2\pi i} \int_{c-iT}^{c+iT} F(s) x^s s^{-1} ds
+ O(x^{1+\eps} T^{-1}).
$$
Moving the contour of the integration till $[1/2 \hm -iT, 1/2 \hm+ iT]$ we get
$$
\sum_{n\le x} \bigl( \taui(n) \bigr)^2
 = \res_{s=1} F(s)x^s s^{-1} + O(I_0 + I_- + I_+ + x^{1+\eps} T^{-1}),
$$
where
$$
I_0 := \int_{1/2-iT}^{1/2+iT} F(s) x^s s^{-1} ds,
\qquad
I_\pm := \int_{1/2\pm iT}^{c\pm iT} F(s) x^s s^{-1} ds.
$$
Function $F(s) x^s s^{-1}$ has a pole of fourth order at $s=1$, so $\res_{s=1} F(s) \* x^s \* s^{-1}$ has form $P_3(\log x) x$. Let us estimate the error term.

Take $T=x^{3/4}$.
Firstly,
$$
I_+ \ll T^{-1} \int_{1/2}^c {\zeta^4(\sigma+iT) \over \zeta^6(2\sigma+2iT)} x^\sigma\, d\sigma
$$
Using classic estimates $\zeta(\sigma+iT) \ll T^{(1-\sigma)/3}$ for $\sigma\in[1/2,1)$ and $\zeta(\sigma\hm+iT) \ll T^\eps$ for $s\ne\sigma\ge1$ we obtain $I_+ \ll x^{1/4+\eps}$. The same can be proven for $I_-$.

Secondly, taking into account bounds $\zeta^{-1}(1+it) \ll \log^{2/3} t$ and $\int_1^T \zeta(1/2+it)^2 t^{-1} dt \ll \log^5 T$ (see \cite{ivic2003} or  \cite{titchmarsh1986}) we have
\begin{equation}\label{eq:I0-estimate}
I_0 \ll x^{1/2} \int_1^T {\zeta^4(1/2+it) \over \zeta^6(1+2it)} {dt\over t} \ll x^{1/2} \log^9 T,
\end{equation}
which completes the proof.
\end{proof}

Recently Jia and Sankaranarayanan  proved in the preprint~\cite{jia2014} that
$$
\int_1^T {\zeta^4(1/2+it) \over \zeta^k(1+2it)} dt \ll T \log^4 T,
$$
so summing up integrals over intervals $[2^n, 2^{n+1}]$ for $n=0,\ldots,\lfloor \log_2 T \rfloor$ leads to
$$
\int_1^T {\zeta^4(1/2+it) \over \zeta^k(1+2it)} {dt\over t} \ll \log^5 T.
$$
Thus instead of~\eqref{eq:I0-estimate} we get $I_0 \ll x^{1/2} \log^5 x$, which provides us with a better error term in~\eqref{eq:taui^2-sum}.

\bigskip

Function $E\taui$ has Dirichlet series $\zeta(s) \zeta(2s) \zeta^{-1}(4s) H(s)$, where $H(s)$ converges absolutely for $\sigma > 1/5$, very similar to function $t^{(e)}$, studied by Tóth~\cite{toth2007b} and Pétermann~\cite{petermann2010}. The latter achieved error term $O(x^{1/4})$ in the asymptotic expansion of $\sum_{n\le x} t^{(e)}$; the same result holds for $E\taui$.

Dirichlet series for $\bigl(E\taui(n)\bigr)^2$ is similar to $\bigl(\taue(n)\bigr)^2$: both of them are~$\zeta(s) \* \zeta^3(2s) H(s)$, where $H(s)$ converges absolutely for $\sigma > 1/3$. Krätzel proved in \cite{kratzel2010} that the asymptotic expansion of $\sum_{n\le x} \bigl(\taue(n)\bigr)^2$ has error term~$O(x^{11/31})$; the same is true for $\bigl(E\taui(n)\bigr)^2$.

\section{E-$\infty$-perfect numbers}\label{s:perfect}

Let $\sigmaei$ denote {\em sum-of-e-$\infty$-divisor function,} where e-$\infty$-divisors were defined in Section~\ref{s:introduction}. We call $n$ e-$\infty$-perfect if $\sigmaei(n) = 2n$. As far as $\sigmaei(n)/n$ depends only on square-full part of~$n$, we consider only square-full~$n$ below. We found following examples of e-$\infty$-perfect numbers:
\begin{multline*}
36,
2700,
1800,
4769\,856,
357\,739\,200,
238\,492\,800,
54\,531\,590\,400, \\
1307\,484\,087\,615\,221\,689\,700\,651\,798\,824\,550\,400\,000.
\end{multline*}
All of them are e-perfect also: $\sigmae(n) = 2n$. We do not know if there are any e-$\infty$-perfect numbers, which are not e-perfect.

Equation~\eqref{eq:taui-primorial} implies that $\taui(n)$ is even for $n\ne 1$. Then for $p>2$, $a>1$ the value of $\sigmaei(p^a)$ is a sum of even number of odd summands and is even. Thus if $n$ is odd and~$\sigmaei(n) \hm= 2n$ then $n=p^a$, $p>2$. But definitely $\sigmaei(p^a) \hm\le \sigma(p^a) < 2p^a$. We conclude that all e-$\infty$-perfect numbers are even.

Are there e-$\infty$-perfect numbers, which are not divisible by 3? For e-perfect numbers Fabrykowski and Subbarao \cite{fabrykowski1986} have obtained that if~$\sigmae(n)\hm=2n$ and~$3\hm\nmid n$ then $n>10^{664}$. We are going to show that in the case of e-$\infty$-perfect even better estimate can be given.

\begin{lemma}
\begin{align}
\nonumber
{\sigmaei(p)^{\phantom1} \over p} &= 1, \\
\nonumber
{\sigmaei(p^2) \over p^2} &= 1+p^{-1}, \\
\label{eq:sigmaei-a6}
{\sigmaei(p^a) \over p^a} &\le 1+2p^{-a/2} \text{~for~} a\ge6, \\
\label{eq:sigmaei-a3}
{\sigmaei(p^a) \over p^a} &\le 1+p^{-2} \text{~for~} a\ge3.
\end{align}
\end{lemma}
\begin{proof}
Two first identities are trivial. For $a\ge 6$ all non-proper divisors of $a$ are less or equal to~$a/2$, so
$$
\sigmaei(p^a) \le p^a + \sum_{b=1}^{a/2} p^b \le p^a + {p(p^{a/2}-1) \over p-1} \le p^a + 2p^{a/2}.
$$
This provides~\eqref{eq:sigmaei-a6}. Inequality~\eqref{eq:sigmaei-a3} can be directly verified for $a\hm=3,4,5$ and follows from~\eqref{eq:sigmaei-a6} for $a\ge6$.
\end{proof}

\begin{lemma}
Let $b(t) = \max_{\tau \ge t} {\sigmaei(2^\tau) 2^{-\tau}}$. Then
\begin{equation}\label{eq:b(t)-estimates}
b(t) \le \begin{cases}
5/4, & t\le3, \\
39/32, & 2<t\le6, \\
1+2^{1-t/2}, & t>6.
\end{cases}
\end{equation}
\end{lemma}
\begin{proof}
Follows from~\eqref{eq:sigmaei-a6} and direct computations for small $\tau$:
$$
\sigmaei(2^3) = 10, \qquad \sigmaei(2^6) = 78.
$$
\end{proof}

\begin{theorem}
If $n$ is e-$\infty$-perfect and $3 \nmid n$ then $n>1.35\cdot10^{816}$.
\end{theorem}
\begin{proof}
In fact we will give a lower estimate for square-full $n$ such that
for~$u/v \hm= \sigmaei(n)/n$, $\gcd(u,v)=1$, we have
\begin{gather}
\label{eq:uv-problem-3}
3\nmid u, \quad
3\nmid v, \\
\label{eq:uv-problem-2}
2\mid u, \quad
4\nmid u, \quad
2\nmid v, \\
\label{eq:uv-problem-ineq}
u/v \ge 2.
\end{gather}
If this conditions are not satisfied then $n$ is not e-$\infty$-perfect or $3\mid n$.

Let
$$
n = 2^t \prod_{p \in P} p^2 \prod_{q \in Q} q^{a_q},
\qquad t\ge1,
\qquad a_q\ge3,
$$
sets $P$ and $Q$ contain primes $\ge5$ and $P \cap Q = \varnothing$. Then
$$
{u\over v} = {\sigmaei(n) \over n}
= {\sigmaei(2^t) \over 2^t}
\prod_{p\in P} {p+1\over p}
\prod_{q\in Q} {\sigmaei(q^{a_q}) \over q^{a_q}}.
$$
Condition~\eqref{eq:uv-problem-3} implies that all $p\in P$ are of form $p=6k+1$. Split~$P$ into three disjoint sets:
\begin{align*}
P_8 &= \{ p\in P \mid p+1\equiv0 \pmod8 \}, \\
P_4 &= \{ p\in P \mid p+1\equiv4 \pmod8 \}, \\
P_2 &= P \setminus P_4 \setminus P_8.
\end{align*}
Let $t_2 = |P_2|$, $t_4=|P_4|$, $t_8=|P_8|$. Then condition~\eqref{eq:uv-problem-2} implies
$$
t \ge t_2 + 2t_4 + 3t_8 + |Q| + 1.
$$
Now we utilize~\eqref{eq:uv-problem-ineq} to get
\begin{multline*}
2 \le {u\over v} \le b(t_2 + 2t_4 + 3t_8 + |Q| + 1) \prod_{p\in P} (1+p^{-1}) \prod_{p \in Q}(1+q^{-2})
=\\= b(t_2 + 2t_4 + 3t_8 + |Q| + 1) \prod_{p\in P} {1+p^{-1} \over 1+p^{-2}} \prod_{q \in P\cup Q}(1+q^{-2}).
\end{multline*}
But
\begin{equation}\label{eq:prod_q_PQ}
\prod_{q \in P\cup Q}(1+q^{-2})
\le {\prod_q (1+q^{-2}) \over (1+2^{-2}) (1+3^{-2})}
= {\zeta(2)/\zeta(4) \over 25/18} = {54\over5\pi^2},
\end{equation}
so we obtain
$$
{10\pi^2 \over 54} \le b(t_2 + 2t_4 + 3t_8 + 1)
\prod_{p\in P} {1+p^{-1} \over 1+p^{-2}}.
$$

Denote $f(p) = (1+p^{-1}) / (1+p^{-2})$. As soon as $f$ is decreasing we can estimate
$$
\prod_{p\in P} f(p) = \prod_{j\in\{2,4,8\} \atop p\in P_j} f(p)
\le
\prod_{k=1}^{t_2} f(p_{2,k})
\prod_{k=1}^{t_4} f(p_{4,k})
\prod_{k=1}^{t_8} f(p_{8,k}),
$$
where $p_2$ is a sequence of consecutive primes such that $p_{2,k}\equiv1\pmod6$ and~$p_{2,k}\hm+1\hm{\not\equiv}0\pmod4$; $p_4$ is a sequence of consecutive primes such that again~$p_{4,k}\equiv1\pmod6$, but~$p_{4,k}+1\hm\equiv4\pmod8$; and $p_8$ is such that $p_{8,k}\equiv1\pmod6$, $p_{8,k}+1\equiv0\pmod8$.

Now conditions~\eqref{eq:uv-problem-3}, \eqref{eq:uv-problem-2}, \eqref{eq:uv-problem-ineq} can be rewritten as
\begin{multline*}
n \ge \min_{t_2, t_4, t_8} \biggl\{
	2^{t_2+2t_4+3t_8+1}
	\prod_{j \in \{2,4,8\}}
	\prod_{k=1}^{t_j} p_{j,k}^2 \biggm|
	\\
	\biggm|
	{10\pi^2 \over 54} \le b(t_2 + 2t_4 + 3t_8 + 1)
	\prod_{j \in \{2,4,8\}}
	\prod_{k=1}^{t_j} {1+p_{j,k}^{-1} \over 1+p_{j,k}^{-2}}
\biggr\}.
\end{multline*}
This optimization problem can be solved numerically utilizing~\eqref{eq:b(t)-estimates}:
$$
t_2 = 70, \quad
t_4 = 32, \quad
t_8 = 31, \quad
n > 8.49\cdot10^{801}.
$$

We can use $n$'s factor $\prod_{q\in Q} q^{a_q}$ to improve obtained bound. Suppose that any of primes 5, 11, 17, 23 (all of form $6k-1$) is not in $Q$. Then instead of~\eqref{eq:prod_q_PQ} we derive
$$
\prod_{q \in P\cup Q}(1+q^{-2})
\le
{54 \over 5 \pi^2 (1+23^{-2})}.
$$
Same arguments as above shows that in this case $n > 3\cdot10^{823}$. Otherwise, if 5, 11, 17 and 23 are present in $Q$ we get
$$
n > (8.49\cdot10^{801}) \cdot (5\cdot11\cdot17\cdot23)^3 \cdot 2^4
> 1.35\cdot10^{816}.
$$
\end{proof}

\bibliographystyle{ugost2008s}
\bibliography{taue}

\end{document}